\input ./style/arxiv-vmsta.cfg
\documentclass[numbers,compress,v1.0.1]{vmsta}
\usepackage{vtexbibtags}

\usepackage{bbm}

\volume{6}% Updated by VTEXPTS2LaTeX.exe, 25.03.2019 11:47
\issue{1}% Updated by VTEXPTS2LaTeX.exe, 25.03.2019 11:47
\pubyear{2019}% Updated by VTEXPTS2LaTeX.exe, 25.03.2019 11:47
\firstpage{13}% Updated by VTEXPTS2LaTeX.exe, 26.03.2019 09:18
\lastpage{39}% Updated by VTEXPTS2LaTeX.exe, 26.03.2019 09:18
\aid{VMSTA126}
\doi{10.15559/18-VMSTA126}% Updated by VTEXPTS2LaTeX.exe, 13.12.2018
\articletype{research-article}

\startlocaldefs

\newcommand{\rrvert}{\vert}
\newcommand{\llvert}{\vert}
\allowdisplaybreaks

\newtheorem{theorem}{Theorem}[section]
\newtheorem{corollary}{Corollary}[section]
\newtheorem{lemma}{Lemma}[section]

\theoremstyle{definition}
\newtheorem{definition}{Definition}[section]
\newtheorem{remark}{Remark}[section]

\DeclareMathOperator*\lowlim{\underline{lim}}

\newcommand{\Y}{\ensuremath{Y_{\varepsilon}}}
\newcommand{\Yn}{\ensuremath{Y_{\varepsilon_n}}}
\newcommand{\fr}[1]{\ensuremath{\frac{ k}{\Y(#1) \mathbbm{1}_{\{\Y(#1)>0\}} + \varepsilon}}}
\newcommand{\frn}[1]{\ensuremath{\frac{k }{\Yn(#1)\mathbbm{1}_{\{\Yn(#1)>0\}} +\varepsilon_n}}}
\newcommand{\B}[1]{\ensuremath{B^H(#1)}}

\newcommand{\E}{\ensuremath{\mathbb E}}

\endlocaldefs

\begin{document}

\begin{frontmatter}
\pretitle{Research Article}

\title{Fractional Cox--Ingersoll--Ross process with small Hurst indices}

\author{\inits{Yu.} \fnms{Yuliya} \snm{Mishura}\ead[label=e1]{myus@univ.kiev.ua}}
\author{\inits{A.} \fnms{Anton} \snm{Yurchenko-Tytarenko}\thanksref{cor1}\ead[label=e2]{antonyurty@gmail.com}}
\thankstext[type=corresp,id=cor1]{Corresponding author.}
\address{Faculty of Mechanics and Mathematics,
\institution{Taras Shevchenko National University~of~Kyiv}, Volodymyrska St., 64/13, Kyiv 01601, \cny{Ukraine}}
%\thankstext[id=f1]{}

%\dedicated{}

%\markboth{Authors}{Title}
\markboth{Yu. Mishura, A. Yurchenko-Tytarenko}{Fractional
Cox--Ingersoll--Ross process with small Hurst indices}

\begin{abstract}
In this paper the fractional Cox--Ingersoll--Ross process on
$\mathbb R_+$ for $H<1/2$ is defined as a square of a pointwise limit of the processes
$\Y$, satisfying the SDE of the form $d\Y(t) =  (\fr{t}-a\Y(t) )dt
+\sigma dB^H(t)$, as $\varepsilon\downarrow0$. Properties of such limit
process are considered. SDE for both the limit process and the fractional
Cox--Ingersoll--Ross process are obtained.
\end{abstract}

\received{\sday{27} \smonth{8} \syear{2018}}% Updated by
%VTEXPTS2LaTeX.exe, 13.12.2018 13:39
\revised{\sday{3} \smonth{12} \syear{2018}}% Updated by
%VTEXPTS2LaTeX.exe, 13.12.2018 13:39
\accepted{\sday{3} \smonth{12} \syear{2018}}% Updated by
%VTEXPTS2LaTeX.exe, 13.12.2018 13:39
\publishedonline{\sday{21} \smonth{12} \syear{2018}}
\end{frontmatter}

\section*{Introduction}

The Cox--Ingersoll--Ross (CIR) process, that was first
introduced and studied by Cox, Ingersoll and Ross in papers \cite{CIR1,CIR2,CIR3},
can be defined as the process \mbox{$X=\{X_t, t\ge0\}$} that
satisfies the stochastic differential equation of the form
\begin{equation}
\label{std cir intro 1} dX_t = a(b-X_t)dt +\sigma
\sqrt{X_t}dW_t, \quad X_0, a, b, \sigma>0,
\end{equation}
where $W=\{W_t, t\ge0\}$ is the Wiener process.

The CIR process\index{CIR process} was originally proposed as a model for interest rate
evolution in time and in this framework the parameters $a$, $b$ and
$\sigma$ have the interpretation as follows: $b$ is the ``mean'' that
characterizes the level, around which the trajectories of $X$ will
evolve in a long run; $a$ is ``the speed of adjustment'' that
measures the velocity of regrouping around $b$; $\sigma$ is an
instantaneous volatility which shows the amplitude of randomness
entering the system. Another common application is the stochastic
volatility modeling in the Heston model which was proposed in \cite
{Hest} (an extensive bibliography on the subject can be found in \cite
{KIM} and \cite{KIMM}).

For the sake of simplicity, we shall use another parametrization of the
equation \eqref{std cir intro 1}, namely
\begin{equation}
\label{std cir intro 2} dX_t = (k-aX_t)dt +\sigma
\sqrt{X_t}dW_t, \quad X_0, k, a, \sigma>0.
\end{equation}

According to \cite{feller}, if the condition $2k \ge\sigma^2$,
sometimes referred to as the Feller condition, holds, then the CIR
process\index{CIR process} is strictly positive. It is also well-known that this process
is ergodic and has a stationary distribution.

However, in reality the dynamics on financial markets is characterized
by the so-called ``memory phenomenon'', which cannot be reflected by
the standard CIR model\index{standard CIR model} (for more details on financial markets with
memory, see \cite{AI,BM,DGE,YMHBS}). Therefore, it is reasonable to
introduce a fractional Cox--Ingersoll--Ross process, i.e. to modify the equation $\eqref{std cir intro 2}$ by replacing (in
some sense) the standard Wiener process with the fractional Brownian
motion\index{fractional Brownian motion} $B^H = \{B_t^H, t\ge0\}$, i.e. with a centered Gaussian process
with the covariance function $\mathbb E  [ B_t^H B_s^H ]
=~\frac{1}{2} (t^{2H}+s^{2H}-|t-s|^{2H} )$.

There are several approaches to the definition of the fractional
Cox--Ingersoll--Ross process. The paper \cite{Mar} introduces the
so-called ``rough-path'' approach; in \cite{LMS1,LMS2} it is defined
as a time-changed standard CIR process\index{CIR process} with inverse stable
subordinator; another definition is presented in \cite{ER} as a part of
discussion on rough Heston models.

A simpler pathwise approach is presented in \cite{MPRYT} and \cite
{MYT}. There, the fractional Cox--Ingersoll--Ross process was defined as
the square of the solution of the SDE
\begin{equation}
\label{Y intro} dY_t = \frac{1}{2} \biggl(\frac{k}{Y_t}-aY_t
\biggr)dt +\frac{\sigma
}{2}dB_t^H, \quad
Y_0>0,
\end{equation}
until the first moment of zero hitting and zero after this moment (the
latter condition was necessary as in the case of $k>0$ the existence of
the solution of \eqref{Y intro} could not be guaranteed after the first
moment of reaching zero).

The reason of such definition lies in the fact that the fractional
Cox--Ingersoll--Ross process $X$ defined in that way satisfies pathwisely
(until the first moment of zero hitting) the equation
\begin{equation*}
dX_t = (k-aX_t)dt + \sigma\sqrt{X_t}\circ
dB_t^H, \quad X_0 = Y_0^2>0,
\end{equation*}
where the integral with respect to the fractional Brownian motion\index{fractional Brownian motion} is
considered as the pathwise Stratonovich integral.\index{pathwise Stratonovich integral}

It was shown that if $k>0$ and $H>\frac{1}{2}$, such process is
strictly positive and never hits zero; if $k>0$ and $H<\frac{1}{2}$,
the probability that there is no zero hitting on the fixed interval
$[0,T]$, $T>0$, tends to 1 as $k\to\infty$.

The special case of $k=0$ was considered in \cite{MPRYT}. In this
situation $Y$ is the well-known fractional Ornstein--Uhlenbeck process
(for more details see \cite{CKM}) and, if $a\ge0$, the fractional
Cox--Ingersoll--Ross process hits zero almost surely, and if $a<0$ the
probability of hitting zero is greater than 0 but less than 1.

However, such a definition has a significant disadvantage: according to
it, the process remains on the zero level after reaching the latter,
and if $H<1/2$, such case cannot be excluded. In this paper we
generalize the approach presented in \cite{MYT} and \cite{MPRYT} in
order to solve this issue.

We define the fractional CIR process\index{fractional CIR process} on $\mathbb R_+$ for $H<1/2$ as
the square of $Y$ which is the pointwise limit\index{pointwise limit} as $\varepsilon
\downarrow0$ of the processes $\Y$ that satisfy the SDE of the
following type:
\begin{equation*}
d\Y(t) = \frac{1}{2} \biggl(\fr{s} - a\Y(t) \biggr)dt +\frac{\sigma}{2}
dB^H(t), \quad\Y(0) = Y_0>0,
\end{equation*}
where $a$, $k$, $\sigma>0$.

We prove that this limit indeed exists, is nonnegative a.s. and is
positive a.e. with respect to the Lebesgue measure a.s. Moreover, $Y$
is continuous and satisfies the equation of the type
\begin{equation*}
\begin{gathered} Y(t) = Y(\alpha) + \frac{1}{2}\int
_\alpha^t \biggl(\frac{k}{Y(s)} - aY(s) \biggr)ds
+ \frac{\sigma}{2} \bigl(B^H(t) - B^H(\alpha) \bigr)
\end{gathered} %
\end{equation*}
for all $t\in[\alpha,\beta]$ where $(\alpha, \beta)$ is any interval of
$Y$'s positiveness.\index{positiveness}

The possibility of getting the equation of the form \eqref{Y intro} is
restricted due to the obscure structure\index{obscure structure} of the set $\{t\ge0~|~Y(t) = 0\}
$ which is connected to the structure of the level sets of the
fractional Brownian motion.\index{fractional Brownian motion} For some results on the latter see, for
example, \cite{Muk}.

This paper is organised as follows.

In Section \ref{section 1} we give preliminary remarks on some results
concerning the existence and uniqueness of solutions of SDEs driven by
an additive fractional Brownian motion,\index{fractional Brownian motion} as well as explain connection of
this paper to \cite{MPRYT} and \cite{MYT}.

In Section \ref{section 2} we construct the square root process as the
limit of approximations. In particular, it contains a variant of the
comparison Lemma and a uniform boundary for all moments of the prelimit
processes.

In Section \ref{section 3} we consider properties of the square root
process. We prove that $Y$ is nonnegative and positive a.e.,
%as well as its continuity.
and also continuous.

In Section \ref{section 4} we give some remarks concerning the equation
for the limit square root process on $\mathbb R_+$. We obtain the
equation for this process with the noise in form of sum of increments
of the fractional Brownian motion\index{fractional Brownian motion} on the intervals of $Y$'s positiveness.\index{positiveness}

Section \ref{section 5} is fully devoted to the pathwise definition and
equation of the fractional Cox--Ingersoll--Ross process. We prove that on
each interval of positiveness\index{positiveness} the process satisfies the CIR SDE with
the integral with respect to a fractional Brownian motion,\index{fractional Brownian motion} considered as
the pathwise Stratonovich integral.\index{pathwise Stratonovich integral} The equation on an arbitrary finite
interval is also obtained, with the integral with respect to a fractional
Brownian motion\index{fractional Brownian motion} considered as the sum of pathwise Stratonovich
integrals\index{pathwise Stratonovich integral} on intervals of positiveness.\index{positiveness}\looseness=1

\section{Fractional Cox--Ingersoll--Ross process until the first moment
of zero hitting}\label{section 1}

Let $B^H = \{B^H(t), t\ge0\}$ be a fractional Brownian motion\index{fractional Brownian motion} with an
arbitrary Hurst index $H\in(0,1)$.\index{Hurst index} Consider the process $Y = \{
Y(t),~t\ge0\}$, such that
\begin{equation}
\label{Y} dY(t) = \frac{1}{2} \biggl(\frac{k}{Y(t)} - aY(t)
\biggr)dt + \frac{\sigma
}{2} dB^H(t), \quad t\ge0,\ Y_0,
a, k, \sigma>0.
\end{equation}

Note that, according to \cite{NOu}, the sufficient conditions that
guarantee the existence and the uniqueness of the strong solution of
the equation
\begin{equation*}
Z(t) = z + \int_0^t f\bigl(s, Z(s)\bigr)ds +
B^H(t),\quad  t\in[0,T],
\end{equation*}
are as follows:

(i) for $H<1/2$: \textit{linear growth condition}:
\begin{equation}
\label{(i)} \big|f(t,z)\big|\le C\bigl(1+|z|\bigr),
\end{equation}
where $C$ is a positive constant;

(ii) for $H>1/2$: \textit{H\"older continuity}:
\begin{equation}
\label{(ii)} \big|f(t_1,z_1) - f(t_2,
z_2)\big| \le C\bigl(|z_1-z_2|^\alpha+|t_1-t_2|^\gamma
\bigr),
\end{equation}
where $C$ is a positive constant, $z_1, z_2 \in\mathbb R$, $t_1, t_2
\in[0,T]$, $1>\alpha>1-1/2H$, $\gamma> H-\frac{1}{2}$.

The function
\begin{equation*}
f(y) = \frac{k}{y} - ay
\end{equation*}
satisfies both (i) and (ii) for all $y\in(\delta, +\infty)$, where
$\delta\in(0,Y_0)$ is arbitrary. Therefore for all $H\in(0,1)$ the
unique strong solution of \eqref{Y} on $[0,T]$ exists until the first
moment of hitting the level $\delta\in(0,Y_0)$ and, from the fact that
$\delta$ is arbitrary, it exists until the first moment of zero hitting.

Let $\tau:= \sup\{t>0~| ~\forall s\in[0,t): Y(s)>0\}$ be the first
moment of zero hitting by $Y$. The fractional Cox-Ingersoll-Ross
process was defined in \cite{MYT} as follows.

\begin{definition}\label{oldCIRdef}
The fractional Cox--Ingersoll--Ross (CIR) process is the process $X=\{
X(t),~t\ge0\}$, such that
\begin{equation}
\label{X} X(t)= Y^2(t)\mathbbm1_{\{t<\tau\}},
\end{equation}
where $Y$ is the solution of the equation \eqref{Y}.

\begin{remark}
Further, $Y$ will be sometimes referred to as the square root process,
as it is indeed a square root of the fractional CIR process.\index{fractional CIR process}
\end{remark}
\end{definition}

It is known (\cite{MYT}, Theorem 2) that if $H>\frac{1}{2}$, the
process $Y$ is strictly positive a.s., therefore the fractional
Cox--Ingersoll--Ross process is simply $Y^2(t)$, $t\ge0$.

However, in the case of $H<\frac{1}{2}$, the process $Y$ may hit zero.
In such situation, according to \eqref{X}, the corresponding
trajectories of the fractional Cox--Ingersoll--Ross process remain in
zero after this moment, which is an undesirable property for any
financial applications.

Our goal is to modify the definition of the fractional CIR process\index{fractional CIR process} in
order to remove the problem mentioned above.

\section{Construction of the square root process on $\mathbb R_+$ as a
limit of $\varepsilon$-approximations}\label{section 2}

Consider the process $\Y=\{\Y(t),t\in[0,T]\}$ that satisfies the
equation of the form
\begin{equation}
\label{new equation} \Y(t)=Y_0 +\int_0^t
\frac{k }{\Y(s)\mathbbm{1}_{\{\Y(s)>0\}}+\varepsilon
}ds - a\int_0^t\Y(s)ds+\sigma
\B{t}, \quad t\ge0,
\end{equation}
where $Y_0, k, a, \sigma>0$ and $B^H=\{\B{t},t\ge0\}$ is a fractional
Brownian motion\index{fractional Brownian motion} with $H\in(0,1/2)$.

\begin{remark}
We will sometimes call $\Y$ an $\varepsilon$-approximation of the
square-root process $Y$.
\end{remark}

For any $T>0$, the function $f_\varepsilon$: $\mathbb R \to\mathbb R$
such that
\begin{equation*}
f_\varepsilon(y)=\frac{k}{y\mathbbm{1}_{\{y>0\}}+\varepsilon} - ay
\end{equation*}
satisfies the conditions \eqref{(i)} and \eqref{(ii)}, therefore the
strong solution of \eqref{new equation} exists and is unique.

The goal of this section is to prove that there is a pointwise limit\index{pointwise limit} of
$\Y$ as $\varepsilon\to0$.

First, let us prove the analogue of the comparison Lemma.

\begin{lemma}\label{comparison} Assume that continuous random processes\index{continuous random processes}
$Y_1 = \{Y_1(t), t\ge0\}$ and $Y_2 = \{Y_2(t), t\ge0\}$ satisfy the
equations of the form
\begin{equation}
\label{comparison processes} %
\begin{gathered} Y_i(t) = Y_0
+\int_0^t f_i\bigl(Y_i(s)
\bigr) ds+\int_0^t \alpha_i(s)ds+
\sigma\B {t}, \quad t \ge0, \  i=1,2, \end{gathered} %
\end{equation}
where $B^H=\{\B t,t\ge0\}$ is a fractional Brownian motion,\index{fractional Brownian motion} $Y_0$,
$\sigma$>0 are constants, $\alpha_i = \{\alpha_i(t), t\ge0\}$,
$i=1,2$, are continuous random processes\index{random processes} and $f_1$, $f_2$: $\mathbb R\to
\mathbb R$ are continuous functions, such that:

(i) for all $y\in\mathbb R$: $f_1(y)<f_2(y)$;

(ii) for all $\omega\in\varOmega$, for all $t\ge0$: $\alpha_1 (t, \omega)
\le\alpha_2(t, \omega)$.

Then, for all $\omega\in\varOmega$, $t\ge0$: $Y_1(t,\omega)<Y_2(t,\omega)$.
\end{lemma}

\begin{proof} Let $\omega\in\varOmega$ be fixed (we will omit $\omega$ in
brackets in what follows).

Denote
\begin{equation*}
\begin{gathered} \delta(t)=Y_2(t)-Y_1(t)
\\
=\int_0^t \bigl(f_2
\bigl(Y_2(s)\bigr)-f_1\bigl(Y_1(s)\bigr)
\bigr)ds +\int_0^t \bigl(\alpha
_2(s)-\alpha_1(s) \bigr)ds, \quad t\ge0. \end{gathered}
\end{equation*}

The function $\delta$ is differentiable, $\delta(0)=0$ and
\begin{equation*}
\delta_+'(0)= \bigl(f_2(Y_0)-f_1(Y_0)
\bigr) + \bigl(\alpha_2(0)-\alpha _1(0) \bigr)>0.
\end{equation*}

It is clear that $\delta(t)=\delta_+'(0)t+o(t)$, $t\rightarrow0+$, so
there exists the maximal interval $(0,t^*)\subset(0, \infty)$ such
that $\delta(t)>0$ for all $t\in(0,t^*)$. It is also clear that
\begin{equation*}
t^*=\sup\bigl\{t>0~|~\forall s\in(0,t): \delta(s)>0\bigr\}.
\end{equation*}

Assume that $t^*< \infty$. \xch{According}{Acording} to the definition of $t^*$, $\delta
(t^*)=0$. Hence $Y_1(t^*)=Y_2(t^*)=Y^*$ and
\begin{equation*}
\delta'\bigl(t^*\bigr)= \bigl(f_2\bigl(Y^*
\bigr)-f_1\bigl(Y^*\bigr) \bigr) + \bigl(\alpha_2
\bigl(t^*\bigr)-\alpha _1\bigl(t^*\bigr) \bigr)>0.
\end{equation*}

As $\delta(t)=\delta'(t^*)(t-t^*)+o(t-t^*)$, $t\rightarrow t^*$, there
exists $\varepsilon>0$ such that $\delta(t)<0$ for all $t\in
(t^*-\varepsilon, t^*)$, that contradicts the definition of $t^*$.

Therefore, $\forall t>0$:
\begin{equation}
\label{comp ineq} Y_1(t)<Y_2(t).\qedhere
\end{equation}
\end{proof}

\begin{remark}\label{comparison remark}
It is obvious that Lemma \ref{comparison} still holds if we replace the
index set $[0,+\infty)$ by $[a,b]$, $a<b$, or if we consider the case
$Y_1(0)<Y_2(0)$.

Moreover, the condition (i) can be replaced by $f_1(y)\le f_2(y)$. In
this case it can be obtained that $Y_1(t,\omega)\le Y_2(t,\omega)$ for
all $\omega\in\varOmega$ and $t\ge0$.
\end{remark}

According to Lemma \ref{comparison}, for any $\varepsilon_1>\varepsilon
_2$ and for all $t\in(0,\infty)$:
\begin{equation*}
Y_{\varepsilon_1}(t)<Y_{\varepsilon_2}(t).
\end{equation*}

Now, let us show that there exists the limit
\begin{equation}
\label{Y defin} \lim_{\varepsilon\to0} \Y(t) = Y(t)<\infty.
\end{equation}

We will need an auxiliary result, presented in \cite{MishurafBm}.
\begin{lemma}\label{fBm moments} For all $r\ge1$:
\begin{equation*}
\mathbb E \Bigl[ \sup_{t\in[0,T]}\big|B^H(t)\big|^r
\Bigr]<\infty.
\end{equation*}
\end{lemma}

Let $a$, $k$, $\sigma>0$, $H\in(0,1)$.
\begin{theorem}\label{upper boundary}
For all $H\in(0,1)$, $T>0$ and $r\ge1$ there are non-random constants
$C_1=C_1(T, r,Y_0, a, k)>0$ and $C_2=C_2(T, r, a, \sigma)>0$ such that
for all $\varepsilon>0$ and for all $t\in[0,T]$:
\begin{equation}
\label{upper boundary formula} \big|\Y(t)\big|^r \le C_1+C_2 \sup
_{s\in[0,T]}\big|\B{s}\big|^r.
\end{equation}
\end{theorem}
\begin{proof}
Let an arbitrary $\omega\in\varOmega$, $r\ge1$ and $T>0$ be fixed (we
will omit $\omega$ in what follows).

Let us prove that for all $\varepsilon>0$ and for all $t \in[0,T]$:
%
%\begin{equation}
\begin{gather}\label{pre-Gronwall} %
\big|\Y(t)\big|^r \le
\biggl((4Y_0)^r + \biggl(\frac{16kT}{Y_0}
\biggr)^r + (8\sigma)^r\sup_{s\in[0,T]}\big|B^H(s)\big|^r
\biggr)\nonumber
\\
+ (8a)^rT^{r-1}\int_{0}^t \big|
\Y(s)\big|^r ds.
\end{gather} %%\end{equation}
Consider the moment
\begin{equation*}
\tau_1 (\varepsilon) := \sup \biggl\{s\in[0,T]~ |~\forall u
\in[0,s]: \Y(s)\ge\frac{Y_0}{2} \biggr\}.
\end{equation*}
Note that $\Y$ is continuous, so $0 < \tau_1 (\varepsilon) \le T$ and,
moreover, for all $t\in[0,\tau_1 (\varepsilon)]$: $\Y(t)\ge\frac{Y_0}{2}>0$.

In order to make the further proof more convenient for the reader, we
shall divide it into 3 steps. In Steps 1 and 2, we will separately show
that \eqref{pre-Gronwall} holds for all $t\in[0,\tau_1 (\varepsilon)]$
and for all $t\in(\tau_1 (\varepsilon), T]$, and in Step 3 we will
obtain the final result.

\textbf{Step 1.} Assume that $t\in[0,\tau_1 (\varepsilon)]$. Then,
%
%\begin{equation}
%
\begin{gather}\label{moments: starting inequality}
\big|\Y(t)\big|^r = \Biggl\llvert
Y_0 + \int_0^t \fr{s}ds - a\int
_0^t \Y(s)ds +\sigma B^H(t) \Biggr
\rrvert ^r\nonumber
\\[-2pt]
\le \Biggl(Y_0 + \int_0^t
\fr{s}ds+ a\int_0^t \big|\Y(s)\big|ds +\sigma
\big|B^H(t)\big| \Biggr)^r.
\end{gather} %
%\end{equation}

If $r>1$, by applying the H\"older inequality to the right-hand side
of \eqref{moments: starting inequality}, we obtain that
%
%\begin{equation}
\begin{gather} \label{true for all r} %
\big|\Y(t)\big|^r
\le4^{r-1} \Biggl(Y_0^r + \Biggl(\int
_0^t \fr{s}ds \Biggr)^r\nonumber
\\[-2pt]
+ a^r \Biggl(\int_0^t \big|\Y(s)\big|ds
\Biggr)^r + \sigma^r\big|B^H(t)\big|^r
\Biggr).
\end{gather} %%\end{equation}
Note that \eqref{true for all r} is also true for $r=1$ as in this case
it simply coincides with right-hand side of \eqref{moments: starting
inequality}.

For all $t\in[0,\tau_1 (\varepsilon)]$
\begin{equation*}
\Biggl(\int_0^t \fr{s}ds \Biggr)^r
\le \Biggl(\int_0^t \frac
{2k}{Y_0}ds
\Biggr)^r \le \biggl(\frac{2kT}{Y_0} \biggr)^r.
\end{equation*}
For $r\ge1$, from Jensen's inequality,
\begin{equation*}
\Biggl(\int_0^t \big|\Y(s)\big| ds
\Biggr)^r \le t^{r-1} \int_0^t
\big|\Y(s)\big|^r ds \le T^{r-1} \int_0^t
\big|\Y(s)\big|^r ds.
\end{equation*}
Finally, for all $t\in[0,\tau_1 (\varepsilon)]$:
\begin{equation*}
\big|B^H(t)\big|^r \le\sup_{s\in[0,T]}\big|B^H(s)\big|^r.
\end{equation*}
Hence, for all $t\in[0,\tau_1 (\varepsilon)]$:
\begin{equation*}
\begin{gathered} \big|\Y(t)\big|^r \le4^{r-1}
\Biggl(Y_0^r + \biggl(\frac{2kT}{Y_0}
\biggr)^r + a^r T^{r-1} \int_0^t
\big|\Y(s)\big|^r ds + \sigma^r\sup_{s\in[0,T]}\big|B^H(s)\big|^r
\Biggr)
\\[-2pt]
\le \biggl((4Y_0)^r + \biggl(\frac{8kT}{Y_0}
\biggr)^r + (4\sigma)^r\sup_{s\in[0,T]}\big|B^H(s)\big|^r
\biggr) + (4a)^r T^{r-1} \int_0^t
\big|\Y(s)\big|^r ds
\\[-2pt]
\le \biggl((4Y_0)^r + \biggl(\frac{16kT}{Y_0}
\biggr)^r + (8\sigma)^r\sup_{s\in[0,T]}\big|B^H(s)\big|^r
\biggr) + (8a)^r T^{r-1} \int_0^t
\big|\Y(s)\big|^r ds. \end{gathered} %
\end{equation*}

Note that the inequality \eqref{pre-Gronwall} holds for all $t\in
[0,T]$, if $\tau_1 (\varepsilon)=T$.\newpage

\textbf{Step 2.} Assume that $\tau_1 (\varepsilon)<T$, i.e. the
interval $(\tau_1 (\varepsilon), T]$ is non-empty. From definition of
$\tau_1 (\varepsilon)$ and continuity\index{continuity} of $\Y$, $\Y(\tau_1 (\varepsilon
)) = \frac{Y_0}{2}$,  for all $t\in(\tau_1 (\varepsilon), T]$:
\begin{equation*}
\biggl\{s\in(\tau_1 (\varepsilon), t]~ |~\big|\Y(s)\big|<\frac{Y_0}{2}
\biggr\} \ne\emptyset.
\end{equation*}

Denote
\begin{equation*}
\tau_2(\varepsilon, t) := \sup \biggl\{s\in(\tau_1 (
\varepsilon), t]~ |~\big|\Y(s)\big|<\frac{Y_0}{2} \biggr\}.
\end{equation*}
Note that for all $t \in(\tau_1 (\varepsilon), T]$: $\tau_1(\varepsilon
)<\tau_2(\varepsilon, t)\le t$ and $|\Y(\tau_2(\varepsilon, t))|\le
\frac{Y_0}{2}$, so
%
%\begin{equation}
\begin{gather}\label{inequality after getting in} %
\big|\Y(t)\big|^r = \big|\Y(t) - \Y
\bigl(\tau_2(\varepsilon, t)\bigr) + \Y\bigl(\tau_2(
\varepsilon , t)\bigr)\big|^r\nonumber
\\[-3pt]
\le2^{r-1} \bigl(\big|\Y(t) - \Y\bigl(\tau_2(\varepsilon, t)
\bigr)\big|^r + \big|\Y\bigl(\tau _2(\varepsilon, t)
\bigr)\big|^r \bigr)\nonumber
\\[-3pt]
\le2^{r-1} \biggl(\big|\Y(t) - \Y\bigl(\tau_2(\varepsilon, t)
\bigr)\big|^r + \biggl(\frac
{Y_0}{2} \biggr)^r \biggr)\nonumber
\\[-3pt]
\le2^{r-1}\big|\Y(t) - \Y\bigl(\tau_2(\varepsilon, t)
\bigr)\big|^r + Y_0^r.
\end{gather} %
%\end{equation}

In addition, if $\tau_2(\varepsilon, t) = t$, then
\begin{equation*}
\big|\Y(t) - \Y\bigl(\tau_2(\varepsilon, t)\bigr)\big|^r = 0,
\end{equation*}
and if $\tau_2(\varepsilon, t) < t$, then
\begin{equation*}
\begin{gathered} \big|\Y(t) - \Y\bigl(\tau_2(\varepsilon, t)
\bigr)\big|^r= \Bigg|\int_{\tau_2(\varepsilon,
t)}^t \fr{s}ds
\\[-3pt]
- a \int_{\tau_2(\varepsilon, t)}^t \Y(s)ds + \sigma
\bigl(B^H(t)-B^H \bigl(\tau_2(\varepsilon, t)
\bigr) \bigr) \Bigg|^r
\\[-3pt]
\le \Biggl(\int_{\tau_2(\varepsilon, t)}^t \fr{s}ds + a \int
_{\tau
_2(\varepsilon, t)}^t \big|\Y(s)\big| ds + \sigma\big|B^H(t)\big|\\
+\sigma \big|B^H \bigl(\tau_2(
\varepsilon, t) \bigr) \big| \Biggr)^r
\\[-3pt]
\le4^{r-1} \Biggl[ \Biggl(\int_{\tau_2(\varepsilon, t)}^t
\fr{s}ds \Biggr)^r + a^r \Biggl(\int_{\tau_2(\varepsilon, t)}^t
\big|\Y(s)\big| ds \Biggr)^r
\\[-3pt]
+ \sigma^r\big|B^H(t)\big|^r + \sigma^r
\big|B^H \bigl(\tau_2(\varepsilon, t) \bigr) \big|^r
\Biggr].
\end{gathered} %
\end{equation*}
In this case, from definition of $\tau_2(\varepsilon, t)$, for all
$s\in[\tau_2(\varepsilon, t), t]$: $\Y(s) \ge\frac{Y_0}{2}$, so
\begin{equation*}
\Biggl(\int_{\tau_2(\varepsilon, t)}^t \fr{s}ds \Biggr)^r
\le \biggl(\frac
{2kT}{Y_0} \biggr)^r.
\end{equation*}

Furthermore, from Jensen's inequality,
\begin{equation*}
\begin{gathered} \Biggl(\int_{\tau_2(\varepsilon, t)}^t \big|
\Y(s)\big|ds \Biggr)^r \le \Biggl(\int_{0}^t
\big|\Y(s)\big|ds \Biggr)^r \le t^{r-1}\int_{0}^t
\big|\Y(s)\big|^r ds
\\[-3pt]
\le T^{r-1}\int_{0}^t \big|
\Y(s)\big|^r ds. \end{gathered} %
\end{equation*}
Next,
\begin{equation*}
\sigma^r\big|B^H(t)\big|^r + \sigma^r
\big|B^H \bigl(\tau_2(\varepsilon, t) \bigr) \big|^r
\le2\sigma^r\sup_{s\in[0,T]}\big|B^H(s)\big|^r.
\end{equation*}
Hence,
%
%\begin{equation}
\begin{gather}\label{inequality after getting out} %
\big|\Y(t) - \Y\bigl(\tau_2(
\varepsilon, t)\bigr)\big|^r\nonumber
\\
\le \biggl(\frac{8kT}{Y_0} \biggr)^r + (4a)^rT^{r-1}
\int_{0}^t \big|\Y(s)\big|^r ds + (4
\sigma)^r\sup_{s\in[0,T]}\big|B^H(s)\big|^r.
\end{gather} %
%\end{equation}

Finally, from \eqref{inequality after getting out} and \eqref
{inequality after getting in}, for all $t \in(\tau_1 (\varepsilon), T]$:
\begin{equation*}
\begin{gathered} \big|\Y(t)\big|^r
\\
\le \biggl(Y_0^r + \biggl(\frac{16kT}{Y_0}
\biggr)^r + (8\sigma)^r\sup_{s\in[0,T]}\big|B^H(s)\big|^r
\biggr) + (8a)^rT^{r-1}\int_{0}^t
\big|\Y(s)\big|^r ds
\\
\le \biggl((4Y_0)^r + \biggl(\frac{16kT}{Y_0}
\biggr)^r + (8\sigma)^r\sup_{s\in[0,T]}\big|B^H(s)\big|^r
\biggr)
\\
+ (8a)^rT^{r-1}\int_{0}^t \big|
\Y(s)\big|^r ds. \end{gathered} %
\end{equation*}

Therefore, \eqref{pre-Gronwall} indeed holds for all $t\in[0,T]$.

\textbf{Step 3.} From \eqref{pre-Gronwall}, by applying the Gr\"onwall
inequality, we obtain that for all $t \in[0, T]$:
\begin{equation*}
\begin{gathered} \big|\Y(t)\big|^r \le \biggl((4Y_0)^r
+ \biggl(\frac{16kT}{Y_0} \biggr)^r + (8\sigma)^r\sup
_{s\in[0,T]}\big|B^H(s)\big|^r \biggr)e^{(8aT)^r}
\\
=: C_1 + C_2 \sup_{s\in[0,T]}\big|B^H(s)\big|^r,
\end{gathered} %
\end{equation*}
where
\begin{gather*}
C_1 = e^{(8aT)^r} \biggl((4Y_0)^r +
\biggl(\frac{16kT}{Y_0} \biggr)^r \biggr),
\\C_2 = (8\sigma)^re^{(8aT)^r},
\end{gather*}
which ends the proof.
\end{proof}

\begin{corollary}
For all $T>0$ and $r\ge1$ there exists $C=C(T, r, Y_0, k, a, \sigma,
H)<\infty$ such that for all $\varepsilon>0$ and $t\in[0,T]$:
\begin{equation*}
\E\big|\Y(t)\big|^r <C.
\end{equation*}
\end{corollary}

\begin{proof}
The proof immediately follows from Lemma \ref{fBm moments} and Theorem
\ref{upper boundary}.
\end{proof}

\begin{corollary}\label{existence of the limit}
Let $r\ge1$ and $C_1$, $C_2$ be constants from Theorem \ref{upper
boundary}, and
\begin{equation}
\label{limit process} Y(t,\omega) = \lim_{\varepsilon\to0}\Y(t,\omega), \quad t
\in[0,T], \ \omega\in\varOmega.
\end{equation}

Then
\begin{equation*}
\big|Y(t)\big|^r < C_1 + C_2\sup
_{s\in[0,T]}\big|\B{s}\big|^r.
\end{equation*}

In particular,
\begin{equation*}
Y(t) <\infty\quad a.s.
\end{equation*}
\end{corollary}
\begin{proof}
Let $\omega\in\varOmega$ and $t\in(0,T]$ be fixed (if $t=0$, the result is
trivial).

From Lemma \ref{comparison}, if $\varepsilon_1>\varepsilon_2$, then
$Y_{\varepsilon_1}(t) < Y_{\varepsilon_2}(t)$, therefore the limit in
\eqref{limit process} exists.

The upper bound for $|Y(t)|^r$ immediately follows from Theorem \ref
{upper boundary} as the right-hand side of \eqref{upper boundary
formula} does not depend on $\varepsilon$. The a.s. finiteness of $Y$
follows from the a.s. finiteness of $\sup_{s\in[0,T]}|\B{s}|$, which is
a direct consequence of Lemma \ref{fBm moments}.
\end{proof}

\begin{corollary}
The process $Y$ is Lebesgue integrable on an arbitrary interval $[0,t]$ a.s.
\end{corollary}

\begin{proof}
First, note that the trajectories of $Y$ are measurable as they are the
pointwise limits\index{pointwise limit} of continuous functions.

Let $t\in\mathbb R_+$ be fixed. Due to Tonelli's theorem, for any
$r\ge1$ there is such $C$ that
\begin{equation*}
\mathbb E \Biggl[ \Biggl\llvert \int_0^t
Y(s)ds \Biggr\rrvert ^r \Biggr]\le T^{r-1} \mathbb E \Biggl[
\int_0^t \bigl\llvert Y(s) \bigr\rrvert
^rds \Biggr] = T^{r-1}\int_0^t
\mathbb E \bigl\llvert Y(s) \bigr\rrvert ^rds \le CT^r,
\end{equation*}
therefore
\begin{equation*}
\int_0^t Y(s)ds < \infty\quad \mathrm{a.s.}
\end{equation*}
\end{proof}

\begin{remark}\label{monotone conv of int} For all $t>0$:
\begin{equation*}
\lim_{\varepsilon\to0}\int_0^t \Y(s)ds =
\int_0^t Y(s)ds \quad \mathrm{a.s.}
\end{equation*}
\end{remark}
\begin{proof}
It follows immediately from monotonicity of $\Y$ with respect to
$\varepsilon$.
\end{proof}

\begin{remark}
Later it will be shown that $Y$ is Riemann integrable\index{Riemann integrable} as well. Until
that, all integrals should be considered as the Lebesgue integrals.
\end{remark}

\begin{remark}
We will sometimes refer to the limit process $Y$\index{limit process} as to the square root
process. It will be shown that it coincides with the square root
process presented in Section \ref{section 1} until the first zero
hitting by the latter.
\end{remark}

\begin{remark}
Further, we will consider only finite and integrable paths of $Y$.
\end{remark}

\section{Properties of $\varepsilon$-approximations and the square root
process}\label{section 3}

Now let us prove several properties of both square root process and its
$\varepsilon$-approxima\-tions.

\begin{lemma}\label{measurelemma} Let $T>0$ and $\lambda$ be the
Lebesgue measure on $[0,T]$. Then
\begin{equation*}
\lambda\bigl\{t\in[0,T]~| ~\Y(t)\le0\bigr\}\to0, \quad\varepsilon\to0,~a.s.
\end{equation*}
\end{lemma}\newpage

\begin{proof} Let $\omega\in\varOmega$ be fixed (we will omit $\omega$ in
what follows).

From the definition of $Y$ and Remark \ref{monotone conv of int}, for any
$t\in[0,T]$ the left-hand side of
\begin{equation*}
Y_{\varepsilon}(t) - Y_0 + a\int_0^t
Y_{\varepsilon}(s)ds - \sigma\B {t} = \int_0^t
\fr{s}ds
\end{equation*}
converges to
\begin{equation*}
Y(t) - Y_0 + a\int_0^t Y(s)ds -
\sigma\B{t}<\infty,
\end{equation*}
as $\varepsilon\to0$. Therefore there exists a limit
\begin{equation}
\label{limit fraction} \lim_{\varepsilon\to0} \int_0^t
\fr{s}ds <\infty.
\end{equation}

Assume that there exists a sequence $\{\varepsilon_n:n\ge1\}$,
$\varepsilon_n\downarrow0$, and $\delta>0$ such that for all $n\ge1$:
\begin{equation*}
\lambda\bigl\{t\in[0,T]~|~\Yn(t)\le0\bigr\}\ge\delta>0.
\end{equation*}

In such case, as
\begin{equation*}
\begin{gathered} \int_0^T \frn{t}\xch{dt}{dt=}
\\
=\int_{\{t\in[0,T]| \Yn(t)\le0\}}\frn{t}dt
\\
+ \int_{\{t\in[0,T]| \Yn(t)> 0\}}\frn{t}dt
\\
\ge\int_{\{t\in[0,T]| \Yn(t)\le0\}}\frn{t}dt \ge\frac{k\delta
}{\varepsilon_n}, \end{gathered}
\end{equation*}
it is clear that
\begin{equation*}
\int_0^T \frn{t}dt \to\infty, \quad n\to
\infty,
\end{equation*}
that contradicts \eqref{limit fraction}.
\end{proof}

\begin{corollary}\label{positiveness}
For any $T>0$, $Y(t)>0$ almost everywhere on $[0,T]$ a.s. and hence
$Y(t)>0$ almost everywhere on $\mathbb R_+$ a.s.
\end{corollary}

\begin{corollary} Let $T>0$ be arbitrary. Then, for all $t\in[0,T]$:
\begin{equation*}
\int_0^t \frac{k}{Y(s)}ds<\infty.
\end{equation*}
\end{corollary}
\begin{proof}
According to the Fatou lemma,
\begin{equation*}
\int_0^t \frac{k}{Y(s)}ds \le
\lowlim_{\varepsilon\to0}\int_0^t \fr {s}ds<\infty.\qedhere
\end{equation*}
\end{proof}

For the next result, we will require the following well-known property
of the fractional Brownian motion\index{fractional Brownian motion} (see, for example, \cite{MishurafBm}).

\begin{lemma}\label{Holder fBm}
Let $\{B^H(t), t\ge0\}$ be a fractional Brownian motion\index{fractional Brownian motion} with the Hurst
index $H$.\index{Hurst index} Then there is such $\varOmega'\subset\varOmega$, $\mathbb P(\varOmega
') =1$, that for all $\omega\in\varOmega'$, $T>0$, $\gamma>0$ and $0\le
s\le t\le T$ there is a positive $C = C(\omega, T, \gamma)$ for which:
\begin{equation*}
\big|B^H(t)-B^H(s)\big| \le C|t-s|^{H-\gamma}.
\end{equation*}
\end{lemma}

\begin{lemma}\label{nonnegativity everywhere}
The process $Y=\{Y(t),~t\ge0\}$ is non-negative a.s., so
\[
\bigl\{t\ge0~|~Y(t)\le0\bigr\}=\bigl\{t\ge0~|~Y(t)=0\bigr\} \quad a.s.
\]
\end{lemma}
\begin{proof}
Let an arbitrary $\omega\in\varOmega'$ from Lemma \ref{Holder fBm} be
fixed and assume that there is such $\tau>0$ that $Y(\tau)\le0$. Then,
for all $\varepsilon>0$:
\begin{equation*}
\Y(\tau)<Y(\tau)\le0.
\end{equation*}

Let an arbitrary $\varepsilon>0$ be fixed. Denote
\begin{equation*}
\begin{gathered} \tau_-(\varepsilon):=\sup\bigl\{t\in(0,\tau)~|~
\Y(t)>0\bigr\},
\\
\tau_+(\varepsilon):=\inf\bigl\{t\in(\tau,\infty)~|~\Y(t)>0\bigr\}.
\end{gathered} %
\end{equation*}
Note that, due to continuity\index{continuity} of $\Y$ and Lemma \ref{measurelemma},
$0<\tau_-(\varepsilon)<\tau< \tau_+(\varepsilon) < \infty$.

It is clear that for all $t\in(\tau_-(\varepsilon), \tau_+(\varepsilon
))$: $\Y(t)<0$, therefore $\mathbbm1_{\{\Y(t)>0\}}=0$ and, due to
Lemma \ref{Holder fBm}, there is such $C>0$ that for all $t\in(\tau
_-(\varepsilon), \tau_+(\varepsilon))$:
%
%\begin{equation}
\begin{gather} \label{nonneg ineq} %
0<\Y(t) = \Y(t)-\Y\bigl(\tau_-(
\varepsilon)\bigr)\nonumber
\\
=\int_{\tau_-(\varepsilon)}^t \frac{k}{\varepsilon}ds - a\int_{\tau_-(\varepsilon)}^t
\Y(s)ds + \sigma\xch{\big(B^H(t)-B^H \bigl(\tau_-(\varepsilon)\bigr)\big)}{(B^H(t)-B^H \bigl(\tau_-(\varepsilon)\bigr)}\nonumber
\\
\ge\frac{k}{\varepsilon} \bigl(t-\tau_-(\varepsilon) \bigr) - C\sigma \bigl(t-
\tau_-(\varepsilon)\bigr)^{H/2}.
\end{gather} %%\end{equation}
It is sufficient to prove that
\begin{equation*}
F(\varepsilon) := \min_{t\in[\tau_-(\varepsilon), \tau_+(\varepsilon)]} \biggl(\frac{k}{\varepsilon}
\bigl(t-\tau_-(\varepsilon) \bigr) - C\sigma \bigl(t-\tau_-(\varepsilon)
\bigr)^{H/2} \biggr) \to0, \quad\varepsilon\to0.
\end{equation*}

Indeed,
\begin{equation}
\label{nonneg deriv} \biggl(\frac{k}{\varepsilon} \bigl(t-\tau_-(\varepsilon) \bigr) - C
\sigma \bigl(t-\tau_-(\varepsilon)\bigr)^{H/2} \biggr)_t'
= \frac{k}{\varepsilon} -\frac
{CH\sigma}{2}\bigl(t-\tau_-(\varepsilon)
\bigr)^{\frac{H-2}{2}}.
\end{equation}
Equating the right-hand side of \eqref{nonneg deriv} to 0 and solving
the equation with respect to~$t$, we obtain
\begin{equation*}
t_* = \tau_-(\varepsilon)+C_1 \varepsilon^{\frac{2}{2-H}},
\end{equation*}
where $C_1=\left(\frac{CH\sigma}{2k}\right)^{\frac{2}{2-H}}$.

It is easy to check that the second derivative of the right-hand side
of \eqref{nonneg ineq} is positive on $(\tau_-(\varepsilon),\tau
_+(\varepsilon))$, so $t_*$ is indeed its point of minimum. Therefore
%
%\begin{equation*}
%
\begin{gather*}
F(\varepsilon) = \frac{k}{\varepsilon} \bigl(t_*-
\tau_-(\varepsilon ) \bigr) - C\sigma\bigl(t_*-\tau_-(\varepsilon)
\bigr)^{H/2} = \frac{k}{\varepsilon
}C_1 \varepsilon^{\frac{2}{2-H}}
- C\sigma\bigl(C_1 \varepsilon^{\frac
{2}{2-H}}\bigr)^{H/2}
\\
= kC_1 \varepsilon^{\frac{H}{2-H}} - CC_1^{H/2}
\sigma\varepsilon^{\frac
{H}{2-H}} \to0, \quad\varepsilon\to0.\qedhere
\end{gather*}
\end{proof}

\begin{remark}
It is clear that for all $t\in[0,T]$:
\begin{equation*}
Y(t) \ge Y_0 + \int_0^t
\frac{k}{Y(s)}ds - a\int_0^t Y(s)ds +\sigma
B_t^H.
\end{equation*}
\end{remark}

\begin{lemma}\label{continuity wrt to eps}
For any $\varepsilon^*>0$ and all $t>0$:
\begin{equation*}
\lim_{\varepsilon\to\varepsilon^*}\Y(t) = Y_{\varepsilon^*}(t).
\end{equation*}
\end{lemma}
\begin{proof}
Indeed, denote
\begin{equation*}
\lim_{\varepsilon\downarrow\varepsilon^*}\Y(t) = Z_+(t) \le Y_{\varepsilon^*}(t).
\end{equation*}
It is clear that for all $t\ge0$, $\Y(t) \uparrow Z_+(t)$, $\varepsilon
\downarrow\varepsilon^*$, so for any $t>0$:
\begin{equation*}
\lim_{\varepsilon\downarrow\varepsilon^*}\int_0^t \Y(s) ds
= \int_0^t Z_+(s)ds.
\end{equation*}
Moreover, for all $\varepsilon>\varepsilon^*$:
\begin{equation*}
\frac{1}{\Y(t)\mathbbm1_{\{\Y(t)>0\}} +\varepsilon} \le\frac
{1}{\varepsilon^*},
\end{equation*}
therefore
\begin{equation*}
\lim_{\varepsilon\downarrow\varepsilon^*} \int_0^t
\frac{1}{\Y
(s)\mathbbm1_{\{\Y(s)>0\}} +\varepsilon}ds = \int_0^t
\frac
{1}{Z_+(s)\mathbbm1_{\{Z_+(s)>0\}} +\varepsilon^*}ds
\end{equation*}
and hence
%
%\begin{equation}
\begin{gather}\label{Zplus} %
 Z_+(t) = \lim_{\varepsilon\downarrow\varepsilon^*}
\Y(t)\nonumber
\\
=Y_0 + \lim_{\varepsilon\downarrow\varepsilon^*} \int_0^t
\frac{k}{\Y
(s)\mathbbm1_{\{\Y(s)>0\}} +\varepsilon}ds\nonumber
\\
-a\lim_{\varepsilon\downarrow\varepsilon^*} \int_0^t
\frac{1}{\Y
(s)\mathbbm1_{\{\Y(s)>0\}} +\varepsilon}ds +\sigma B^H(t)\nonumber
\\
=Y_0+\int_0^t
\frac{1}{Z_+(s)\mathbbm1_{\{Z_+(s)>0\}} +\varepsilon
^*}ds - a\int_0^t Z_+(s)ds +
\sigma B^H(t).
\end{gather} %
%\end{equation}

It is known that $Y_{\varepsilon^*}$ is the unique solution to the
equation \eqref{Zplus}, therefore for all $t\ge0$:
\begin{equation*}
\lim_{\varepsilon\downarrow\varepsilon^*}\Y(t) = Y_{\varepsilon^*}(t).
\end{equation*}

Next, denote
\begin{equation*}
\lim_{\varepsilon\uparrow\varepsilon^*}\Y(t) = Z_-(t) \ge Y_{\varepsilon^*}(t).
\end{equation*}
For all $t\ge0$, $\Y(t) \downarrow Z_-(t)$, $\varepsilon\uparrow
\varepsilon^*$, so
\begin{equation*}
\lim_{\varepsilon\uparrow\varepsilon^*}\int_0^t \Y(s) ds
= \int_0^t Z_-(s)ds
\end{equation*}
and for all $\varepsilon\in (\frac{\varepsilon^*}{2}, \varepsilon
^* )$:
\begin{equation*}
\frac{1}{\Y(t)\mathbbm1_{\{\Y(t)>0\}} +\varepsilon} \le\frac
{2}{\varepsilon^*},
\end{equation*}
so
\begin{equation*}
\lim_{\varepsilon\uparrow\varepsilon^*} \int_0^t
\frac{1}{\Y(s)\mathbbm
1_{\{\Y(s)>0\}} +\varepsilon}ds = \int_0^t
\frac{1}{Z_-(s)\mathbbm1_{\{
Z_-(s)>0\}} +\varepsilon^*}ds.
\end{equation*}

Therefore, similarly to \eqref{Zplus}, $Z_-$ satisfies the same
equation as $Y_{\varepsilon^*}$, so
\begin{equation*}
\lim_{\varepsilon\uparrow\varepsilon^*}\Y(t) = Y_{\varepsilon^*}(t).\qedhere
\end{equation*}
\end{proof}

\begin{theorem}\label{a.e. continuity}
Let $Y=\{Y(t), t \ge0\}$ be the process defined by the formula \eqref
{limit process}. Then

1) the set $\{t > 0 | Y(t)>0\}$ is open in the natural topology\index{natural topology} on
$\mathbb R$;

2) $Y$ is continuous on $\{t\ge0~|~Y(t)>0\}$.
\end{theorem}

\begin{proof}

We shall divide the proof into 3 steps.

\textbf{Step 1.} Let $\omega\in\varOmega$ be fixed. Consider an arbitrary
$t^* \in\{t > 0 | Y(t)>0\}$. As $\Y(t^*) \to Y(t^*)$, $\varepsilon\to
0$, there exists such $\varepsilon^*>0$ that for all $\varepsilon
<\varepsilon^*$: $\Y(t^*)>0$. From continuity\index{continuity} of $\Y$ with respect to
$t$ and their monotonicity with respect to $\varepsilon$, it follows
that there exists such $\delta^* = \delta^*(t^*)>0$ that
\begin{equation*}
\forall\varepsilon<\varepsilon^*, ~\forall t \in\bigl(t^*-\delta^*, t^*+\delta^*
\bigr): \quad\Y(t)>0.
\end{equation*}

Hence for all $t \in(t^*-\delta^*, t^*+\delta^*)$: $Y(t)>0$ and
therefore the set $\{t > 0 | Y(t)>0\}$ is open.

\textbf{Step 2.} Let us prove that
\begin{equation*}
\sup_{t\in(t^*-\delta^*, t^*+\delta^*)} \bigl(Y(t) - Y_{\varepsilon
}(t) \bigr)\to0, \quad
\varepsilon\to0,
\end{equation*}
and therefore $Y$ is continuous on the interval $(t^*-\delta^*,
t^*+\delta^*)$.

It is enough to prove that for any $\theta>0$ there exists such
$\varepsilon_0=\varepsilon_0(\theta)>0$ that for all $\varepsilon
<\varepsilon_0$:
\begin{equation*}
\sup_{t\in(t^*-\delta^*, t^*+\delta^*)} \bigl(Y(t) - Y_{\varepsilon
}(t) \bigr)\le\theta.
\end{equation*}

Indeed, let us fix an arbitrary $\theta>0$. From the definition of $Y$
it follows that $Y_\varepsilon(t^*-\delta^*) \to Y (t^*-\delta^*)$ as
$\varepsilon\to0$, so there is such $\varepsilon^{**}<\varepsilon^*$
that for all $\varepsilon<\varepsilon^{**}$ the following inequality holds:
\begin{equation*}
Y \bigl(t^*-\delta^*\bigr) - Y_\varepsilon\bigl(t^*-\delta^*\bigr) <\theta.
\end{equation*}

Denote
\begin{equation*}
\begin{aligned} \varepsilon_1 := \min \bigl\{\theta, ~
\sup\bigl\{\varepsilon^{**} \in\bigl(0, \varepsilon^*\bigr)~|~\forall
\varepsilon\in\bigl(&0,\varepsilon^{**}\bigr):
\\
&Y \bigl(t^*-\delta^*\bigr) - Y_\varepsilon\bigl(t^*-\delta^*\bigr) <\theta
\bigr\} \bigr\}. \end{aligned} %
\end{equation*}
As $\varepsilon_1 \le\theta$, there exists such $C\in(0,1]$ that
$\varepsilon_1 = C\theta$.

From the continuity\index{continuity} with respect to $\varepsilon$,
\begingroup
\abovedisplayskip=8.5pt
\belowdisplayskip=8.5pt
\begin{equation*}
Y \bigl(t^*-\delta^*\bigr) - Y_{\varepsilon_1} \bigl(t^*-\delta^*\bigr) \le\theta
\end{equation*}
and, from the monotonicity with respect to $\varepsilon$, for all
$\varepsilon<\varepsilon_1$:
\begin{equation*}
0 \le Y_\varepsilon\bigl(t^*-\delta^*\bigr) - Y_{\varepsilon_1} \bigl(t^*-
\delta^*\bigr) \le\theta.
\end{equation*}

It is obvious that $Y_\varepsilon(t^*-\delta^*) - Y_{\varepsilon_1}
(t^*-\delta^*) \downarrow0$ as $\varepsilon\uparrow\varepsilon_1$, so
let us denote
\begin{equation*}
\begin{aligned} \varepsilon_0 := \sup \biggl\{
\varepsilon\in(0,\varepsilon_1)~ |~Y_\varepsilon\bigl(t^*-
\delta^*\bigr) &- Y_{\varepsilon_1}\bigl(t^*-\delta^*\bigr)
\\
&\ge\frac{Y (t^*-\delta^*) - Y_{\varepsilon_1} (t^*-\delta^*)}{2} \biggr\}. \end{aligned} %
\end{equation*}

It is obvious that
\begin{equation*}
Y_{\varepsilon_0}\bigl(t^*-\delta^*\bigr) - Y_{\varepsilon_1}\bigl(t^*-\delta^*
\bigr) = \frac{Y (t^*-\delta^*) - Y_{\varepsilon_1} (t^*-\delta^*)}{2}
\end{equation*}
and therefore
\begin{equation*}
\begin{gathered} Y \bigl(t^*-\delta^*\bigr) - Y_{\varepsilon_0}
\bigl(t^*-\delta^*\bigr)
\\
= \bigl(Y \bigl(t^*-\delta^*\bigr) - Y_{\varepsilon_1} \bigl(t^*-\delta^*\bigr)
\bigr) - \bigl(Y_{\varepsilon_0}\bigl(t^*-\delta^*\bigr) - Y_{\varepsilon_1}
\bigl(t^*-\delta^*\bigr) \bigr)
\\
= \frac{Y (t^*-\delta^*) - Y_{\varepsilon_1} (t^*-\delta^*)}{2} \le \frac{\theta}{2}. \end{gathered} %
\end{equation*}

Moreover, for all $\varepsilon<\varepsilon_0$:
\begin{equation*}
Y_\varepsilon\bigl(t^*-\delta^*\bigr) - Y_{\varepsilon_0}\bigl(t^*-\delta^*
\bigr) \le\frac
{\theta}{2}.
\end{equation*}

Now consider an arbitrary $\varepsilon<\varepsilon_0$ and assume that
there is such $\tau_0\in(t^*-\delta^*, t^*+\delta^*)$ that
\begin{equation*}
Y_\varepsilon(\tau_0) - Y_{\varepsilon_0}(\tau_0)>
\theta.
\end{equation*}

Denote
\begin{equation*}
\tau:= \inf\bigl\{t\in\bigl(t^*-\delta^*, \tau_0\bigr)~|~\forall s
\in(t,\tau_0): ~Y_\varepsilon(s) - Y_{\varepsilon_0}(s)>\theta
\bigr\}<\tau_0.
\end{equation*}
From the definition of $\tau_0$ and $\tau$, for all $t\in(\tau,\tau_0)$:
\begin{equation}
\label{inequality for uniform conv} Y_\varepsilon(t) - Y_{\varepsilon_0}(t)>\theta.
\end{equation}

However, as for all $t\in(t^*-\delta^*, t^*+\delta^*)$ and for all
$\varepsilon<\varepsilon^*$ it is true that $\mathbbm1_{\{
Y_{\varepsilon}(t)>0\}}=1$,
%
%\begin{equation}
\begin{gather}\label{delta der 1} %
\bigl(Y_\varepsilon(\tau) -
Y_{\varepsilon_0}(\tau) \bigr)'\nonumber
\\
= \biggl(\frac{k}{Y_{\varepsilon}(\tau)\mathbbm1_{\{Y_{\varepsilon}(\tau
)>0\}}+\varepsilon} - \frac{k}{Y_{\varepsilon_2}(\tau)\mathbbm1_{\{
Y_{\varepsilon_0}(\tau)>0\}}+\varepsilon_0} \biggr)\nonumber
\\
- a \bigl(Y_{\varepsilon}(\tau) - Y_{\varepsilon_0}(\tau) \bigr)\nonumber
\\
= \biggl(\frac{k}{Y_{\varepsilon}(\tau)+\varepsilon} - \frac
{k}{Y_{\varepsilon_0}(\tau)+\varepsilon_0} \biggr) - a
\bigl(Y_{\varepsilon}(\tau) - Y_{\varepsilon_0}(\tau) \bigr)\nonumber
\\
= \frac{k(Y_{\varepsilon_0}(\tau) - Y_{\varepsilon}(\tau)) +
k(\varepsilon_0-\varepsilon)}{(Y_{\varepsilon}(\tau)+\varepsilon
)(Y_{\varepsilon_0}(\tau)+\varepsilon_0)} - a \bigl(Y_{\varepsilon}(\tau ) - Y_{\varepsilon_0}(\tau)
\bigr) .
\end{gather}
\endgroup%
%\end{equation}

From the continuity\index{continuity} of $Y_{\varepsilon_0}(t) - Y_{\varepsilon}(t)$ with
respect to $t$ and definition of $\tau$, it is clear that
$Y_{\varepsilon_0}(\tau) - Y_{\varepsilon}(\tau) = -\theta$ and, as
$0<\varepsilon<\varepsilon_0<\varepsilon_1=C\theta$, where $C\in(0,1]$,
\begin{equation*}
k\bigl(Y_{\varepsilon_0}(\tau) - Y_{\varepsilon}(\tau)\bigr) + k(\varepsilon
_0-\varepsilon) < k(-\theta+C\theta) = k(C-1)\theta\le0.
\end{equation*}

Therefore
\begin{equation*}
\bigl(Y_\varepsilon(\tau) - Y_{\varepsilon_0}(\tau) \bigr)' <
\frac
{k(C-1)\theta}{(Y_{\varepsilon}(\tau)+\varepsilon)(Y_{\varepsilon
_0}(\tau)+\varepsilon_0)} - a\theta<0.
\end{equation*}

Hence, as
\begin{equation*}
Y_\varepsilon(t) - Y_{\varepsilon_0}(t) = \theta+ \bigl(Y_\varepsilon(
\tau ) - Y_{\varepsilon_0}(\tau)\bigr)'(t-\tau) + o(t-\tau), \quad t
\to\tau,
\end{equation*}
there exists such interval $(\tau, \tau_1)\subset(\tau, \tau_0)$ that
for all $t\in(\tau, \tau_1)$:
\begin{equation*}
Y_\varepsilon(t) - Y_{\varepsilon_0}(t) < \theta,
\end{equation*}
which contradicts \eqref{inequality for uniform conv}.

So, for all $t\in(t^*-\delta^*, t^*+\delta^*)$:
\begin{equation*}
Y_\varepsilon\bigl(t^*-\delta^*\bigr) - Y_{\varepsilon_0}\bigl(t^*-\delta^*
\bigr)\le\theta,
\end{equation*}
and hence, as for all $\varepsilon<\varepsilon_0$ and $t\in\mathbb R$
it holds that $Y_{\varepsilon_0}(t) < Y_{\varepsilon}(t) <Y(t)$,
\begin{equation*}
\sup_{t\in(t^*-\delta^*, t^*+\delta^*)} \bigl(Y_\varepsilon(t) - Y_{\varepsilon}(t)
\bigr)\le\theta, \quad\forall\varepsilon<\varepsilon_0.
\end{equation*}

\textbf{Step 3.} In order to prove that
\begin{equation*}
\lim_{t\to0+} Y(t) = Y(0) = Y_0,
\end{equation*}
it is enough to notice that for any $\tilde\varepsilon>0$ there is such
$\tilde t >0$ that for all $\varepsilon<\tilde\varepsilon$: $\Y(t)>\frac
{Y_0}{2}$, $t\in[0,\tilde t]$.

Hence, for each all $\varepsilon<\tilde\varepsilon$ and $t\in[0,\tilde t]$
\begin{equation*}
\fr{t} = \frac{k}{\Y(t)+\varepsilon} \le\frac{k}{\Y(t)} \le\frac{2k}{Y_0}
\end{equation*}
and so
\begin{equation*}
\lim_{\varepsilon\to0}\int_0^{t} \fr{s}ds
= \int_0^{t} \frac{k}{Y(s)}ds,
\end{equation*}
hence, for all $t\in[0,\tilde t]$:
\begin{equation*}
Y(t) = Y_0 + \int_0^{t}
\frac{k}{Y(s)}ds - a\int_0^t Y(s) ds +\sigma
B^H(t).
\end{equation*}

This equation has a unique continuous solution, therefore $Y$ is
continuous on $[0,\tilde t]$.
\end{proof}

\begin{remark}\label{new Y is old Y}
From Theorem \ref{a.e. continuity} it is easy to see that the limit
square root process $Y$ satisfies the equation of the form \eqref{Y}
until the first moment $\tau$ of zero hitting. Indeed, on each compact
set $[0,\tilde t]\subset[0, \tau)$ $\Y$ converges to $Y$ uniformly as
$\varepsilon\to0$ due to Dini's theorem. Hence there is such $\tilde
\varepsilon>0$ that for all $\varepsilon<\tilde\varepsilon$: $\Y
(t)>\frac{\min_{s\in[0,\tilde t]}Y(s)}{2}>0$ for all $t\in[0,\tilde t]$, and, similarly to Step 3 of
Theorem \ref{a.e. continuity}, it can be shown that $Y$ satisfies
equation \eqref{Y} on $[0, \tilde t]$.
\end{remark}

\begin{corollary}
The trajectories of the process $Y = \{Y(t),~ t\ge0\}$ are continuous
a.e. on $\mathbb R_+$ a.s. and therefore are Riemann integrable\index{Riemann integrable} a.s.
\end{corollary}
\begin{proof}
The claim follows directly from Theorem \ref{a.e. continuity} and
Corollary \ref{positiveness}.
\end{proof}

The set $\{t > 0~ | ~Y(t)>0\}$ is open in the natural topology\index{natural topology} on
$\mathbb R$ a.s., so it can be a.s. represented as the finite or
countable union of non-intersecting \xch{intervals}{intevals}, i.e.
\begin{equation*}
\bigl\{t > 0 ~| ~Y(t)>0\bigr\} = \bigcup_{i=0}^N
(\alpha_i, \beta_{i}), \quad N\in\mathbb N\cup\{\infty\},
\end{equation*}
where $(\alpha_i, \beta_{i})\cap(\alpha_j, \beta_{j}) = \emptyset$,
$i\ne j$.

Moreover, the set $\{t\ge0~|~Y(t) >0\}$ can be a.s. represented as
\begin{equation}
\label{pos set repr} \bigl\{t\ge0~|~Y(t) >0\bigr\} = [\alpha_0,
\beta_0)\cup \Biggl(\bigcup_{i=1}^N
(\alpha_i, \beta_{i}) \Biggr),
\end{equation}
where $\alpha_0=0$, $\beta_0$ is the first moment of zero hitting by
the square root process $Y$, $(\alpha_i, \beta_{i})\cap(\alpha_j, \beta
_{j}) = \emptyset$, $i\ne j$, and $(\alpha_i, \beta_{i})\cap[\alpha_0,
\beta_{0}) = \emptyset$, $i\ne0$.

\begin{theorem}\label{lims and eqs}
Let $(\alpha_i, \beta_i)$, $i\ge1$, be an arbitrary interval from the
representation \eqref{pos set repr}. Then

1) $\lim_{t\to\alpha_i +} Y(t) = 0$, $\lim_{t\to\beta_i -} Y(t) = 0$ a.s.;

2) for any $t\in[\alpha_i, \beta_i]$:
\begin{equation}
\label{equation on compact} Y(t) = \int_{\alpha_i}^t
\frac{k}{Y(s)}ds - a\int_{\alpha_i}^t Y(s) ds + \sigma
\bigl(B^H(t)-B^H(\alpha_i) \bigr) \quad
\text{a.s.}
\end{equation}
\end{theorem}

\begin{proof}
Let $\varOmega'$ be from Lemma \ref{Holder fBm} and an arbitrary $\omega\in
\varOmega'$ be fixed.

1) Proofs for both left and right ends of the segment are similar, so
we shall give a proof for the left end only.

$Y$ is positive on $(\alpha_i, \beta_i)$, so it is sufficient to prove
that $\varlimsup_{t \to\alpha_i+} Y(t) = 0$.

Assume that $\varlimsup_{t \to\alpha_i+} Y(t) = x>0$. Then for any
$\delta>0$ there exists such $\tau\in(\alpha_i, \alpha_i+\delta)$
that $Y(\tau)\in (\frac{3x}{4}, \frac{5x}{4} )$.

Let $\delta$ and such $\tau\in(\alpha_i, \alpha_i+\delta)$ be fixed.
$\Y(\tau) \uparrow Y(\tau)$ as $\varepsilon\to0$, so there is such
$\varepsilon= \varepsilon(\tau)$ that $\Y(\tau)\in (\frac{x}{2},
\frac{5x}{4} )$. It is clear that $\Y(\alpha_i)<0$, therefore
there is such a moment $\tau_1\in(\alpha_i, \tau)$ that
\begin{equation*}
\tau_1 = \sup \biggl\{t\in(\alpha_i, \tau)~|~\Y(t)=
\frac{x}{4} \biggr\}.
\end{equation*}

From continuity\index{continuity} of $\Y$, $\Y(\tau_1) = \frac{x}{4}$, so $\Y(\tau)-\Y
(\tau_1)>\frac{x}{4}$. On the other hand, from definitions of $\tau$
and $\tau_1$, for all $t\in[\tau_1, \tau]$: $\Y(t) \in (\frac
{x}{4}, \frac{5x}{4} )$. That, together with Lemma \ref{Holder
fBm}, gives:
\begin{equation*}
\begin{gathered} \frac{x}{4} < \Y(\tau) - \Y(
\tau_1)
\\
= \int_{\tau_1}^{\tau} \frac{k}{\Y(s)+\varepsilon}ds - a\int
_{\tau
_1}^{\tau} \Y(s) ds + \sigma \bigl(B^H(
\tau) - B^H(\tau_1) \bigr)
\\
\le \biggl(\frac{2k}{x} + \frac{ax}{4} \biggr) (\tau-
\tau_1) + C(\tau _2-\tau_1)^{H/2},
\end{gathered} %
\end{equation*}
i.e. for any $\delta>0$
\begin{equation*}
0<\frac{x}{4} \le \biggl(\frac{2k}{x} + \frac{ax}{4} \biggr)
(\tau-\tau_1) + C(\tau_2-\tau_1)^{H/2}
\le \biggl(\frac{2k}{x} + \frac{ax}{4} \biggr)\delta+ C
\delta^{H/2},
\end{equation*}
which is not possible.

Therefore
\begin{equation*}
\varlimsup_{t \to\alpha_i+} Y(t) = \varliminf_{t \to\alpha_i+} Y(t) = \lim
_{t \to\alpha_i+}Y(t) = 0.
\end{equation*}

2) From Theorem \ref{a.e. continuity}, $Y$ is continuous on each
segment $[\alpha_i^*, \beta_i^*]\subset[\alpha_i, \beta_i]$, so, due
to Dini's theorem, $\Y$ converges uniformly to $Y$ on $[\alpha_i^*,
\beta_i^*]$ as $\varepsilon\to0$. Moreover, there is such $\delta>0$
such that for all $t\in[\alpha_i^*, \beta_i^*]$: $Y(t)>\delta$,
therefore it is easy to see that $\fr{\cdot}$ converges uniformly to
$\frac{k}{Y(\cdot)}$ on $[\alpha_i^*, \beta_i^*]$ as $\varepsilon\to0$.

Hence, for all $t\in[\alpha_i^*, \beta_i^*]$:
\begin{equation*}
\begin{gathered} \int_{\alpha_i^*}^t
\frac{1}{Y(s)}ds = \lim_{\varepsilon\to0}\int_{\alpha_i^*}^t
\fr{s}ds
\\
= \lim_{\varepsilon\to0} \Biggl(\Y(t)-\Y\bigl(\alpha_i^*
\bigr)+a\int_{\alpha
_i^*}^t \Y(s)ds -\sigma
\bigl(B^H(t)-B^H\bigl(\alpha^*_i\bigr) \bigr)
\Biggr)
\\
= Y(t) - Y\bigl(\alpha_i^*\bigr) + a\int_{\alpha_i^*}^t
Y(s)ds - \sigma \bigl(B^H(t)-B^H\bigl(
\alpha^*_i\bigr) \bigr), \end{gathered} %
\end{equation*}
or
\begin{equation}
\label{alpha star} Y(t) = Y\bigl(\alpha_i^*\bigr)+\int
_{\alpha_i^*}^t \frac{k}{Y(s)}ds - a\int
_{\alpha^*_i}^t Y(s) ds + \sigma \bigl(B^H(t)-B^H
\bigl(\alpha^*_i\bigr) \bigr).
\end{equation}

The right-hand side of \eqref{alpha star} is right continuous with
respect to $\alpha_i^*$ due to the previous clause of proof, i.e.
\begin{equation*}
\begin{gathered} \lim_{\alpha_i^*\to\alpha_i+} Y\bigl(
\alpha_i^*\bigr) = Y(\alpha_i);\qquad\lim
_{\alpha_i^*\to\alpha_i+} \int_{\alpha_i^*}^t
\frac{k}{Y(s)}ds = \int_{\alpha_i}^t
\frac{k}{Y(s)}ds;
\\
\lim_{\alpha_i^*\to\alpha_i+} \int_{\alpha^*_i}^t Y(s) ds
= \int_{\alpha
_i}^t Y(s) ds;
\\
\lim_{\alpha_i^*\to\alpha_i+} \bigl(B^H(t)-B^H\bigl(
\alpha^*_i\bigr) \bigr) = \bigl(B^H(t)-B^H(
\alpha_i) \bigr). \end{gathered} %
\end{equation*}

Due to Lemma \ref{nonnegativity everywhere}, $Y(\alpha_i)=0$, therefore
\eqref{equation on compact} holds for an arbitrary $t\in[\alpha_i, \beta_i)$.

To get the result for $t=\beta_i$, it is sufficient to pass to the
limit as $t\to\beta_i$.
\end{proof}

\begin{remark}\label{eq until zero}
Similarly to Theorem \ref{lims and eqs}, it is easy to prove that
\begin{equation*}
\lim_{t\to\beta_0-} Y(t) = Y(\beta_0) = 0,
\end{equation*}
and therefore, taking into account Remark \ref{new Y is old Y}, for all
$t\in[0, \beta_0]$:
\begin{equation}
\label{Y until the fzh} Y(t) = Y_0 +\int_0^t
\frac{k}{Y(s)}ds - a\int_0^t Y(s)ds + \sigma
B^H (t).
\end{equation}
\end{remark}

\begin{remark}
The choice of $\varepsilon$-approximations may be different. For
example, instead of \eqref{new equation}, it is possible to consider
the equation of the form
\begin{equation*}
\tilde\Y(t) = Y_0 + \int_0^t
\frac{k}{\max \{\tilde\Y(s),
\varepsilon \}}ds - a\int_0^t \tilde\Y(s)ds +
\sigma B^H(t).
\end{equation*}

By following the proofs of the results above, it can be verified that
all of them hold for the resulting limit process $\tilde Y$.
Furthermore, if $k,a>0$, it coincides with $Y$ on $[0,+\infty)$.

Indeed, let $\omega\in\varOmega$ be fixed. Consider the difference
\begin{equation*}
\begin{gathered} \Delta_\varepsilon(t) := \tilde\Y(t) - \Y(t)
\\
= \int_0^t \biggl(\frac{k}{\max \{\tilde\Y(s), \varepsilon \}
}-\fr{s}
\biggr)ds
\\
- a\int_0^t \bigl(\tilde\Y(s)-\Y(s)\bigr)ds.
\end{gathered} %
\end{equation*}

As $\frac{k}{\max \{x, \varepsilon \}}\ge\frac{k}{x\mathbbm
{1}_{\{x>0\}}+\varepsilon}$ for all $x\in\mathbb{R}$, it is easy to see
from Lemma \ref{comparison} and Remark \ref{comparison remark} that
$\Delta_\varepsilon(t) = \tilde\Y(t) - \Y(t) \ge0$. Furthermore,
$\Delta_\varepsilon$ is differentiable on $(0,+\infty)$ and $\Delta
_\varepsilon(0)=0$.

Assume that there is such $\tau>0$ that $\Delta_\varepsilon(\tau) \ge
2\varepsilon$ and denote
\begin{equation*}
\tau_\varepsilon:= \inf\bigl\{t\in(0,\tau)~| ~ \forall s\in(t,\tau]: \Delta
_\varepsilon(s)>\varepsilon\bigr\}.
\end{equation*}

Note that, due to continuity\index{continuity} of $\tilde\Y$ and $\Y$, $\Delta_\varepsilon
(\tau_\varepsilon) = \varepsilon$ and therefore for all $t\in(\tau
_\varepsilon, \tau)$:
\begin{equation*}
\Delta_\varepsilon(t) - \Delta_\varepsilon(\tau_\varepsilon) > 0,
\end{equation*}
so, as $\Delta_\varepsilon$ is differentiable in $\tau_\varepsilon$,
\begin{equation*}
\Delta_\varepsilon' (\tau_\varepsilon) = (
\Delta_\varepsilon)_+' (\tau _\varepsilon) = \lim
_{t\to\tau_\varepsilon+}\frac{\Delta_\varepsilon
(t) - \Delta_\varepsilon(\tau_\varepsilon)}{t-\tau_\varepsilon} \ge0.
\end{equation*}

However, as
\begin{equation*}
\begin{gathered} \max \bigl\{\tilde\Y(\tau_\varepsilon),
\varepsilon \bigr\} = \max \bigl\{ \Y(\tau_\varepsilon)+\varepsilon,
\varepsilon \bigr\} = \max \bigl\{\Y (\tau_\varepsilon), 0 \bigr\} +
\varepsilon
\\
= \Y(\tau_\varepsilon)\mathbbm{1}_{\{\Y(\tau_\varepsilon)>0\}
}+\varepsilon, \end{gathered}
\end{equation*}
it is easy to see that
\begin{equation*}
\begin{gathered} \Delta_\varepsilon' (
\tau_\varepsilon) = \frac{k}{\max \{\tilde\Y
(\tau_\varepsilon), \varepsilon \}}-\fr{\tau_\varepsilon} - a\bigl(
\tilde\Y(\tau_\varepsilon)-\Y(\tau_\varepsilon)\bigr)
\\
=-a\varepsilon< 0. \end{gathered} %
\end{equation*}

The obtained contradiction shows that for all $t\ge0$: $\tilde\Y(t) -
\Y(t)<2\varepsilon$ and therefore, by letting $\varepsilon\to0$, it is
easy to verify that for all $t\ge0$: $\tilde Y(t) = Y(t)$.

Simulations illustrate that the processes indeed coincide (see Fig. \ref
{different Ys}). Euler approximations\index{Euler approximations} of $\Y$ and $\tilde\Y$ on $[0,5]$
were used with $\varepsilon= 0.01$, $H=0.3$, $Y_0=a=k=\sigma=1$. The
mesh\index{mesh} of the partition was $\Delta t = 0.0001$.
\end{remark}

\begin{figure}[h!]
\includegraphics{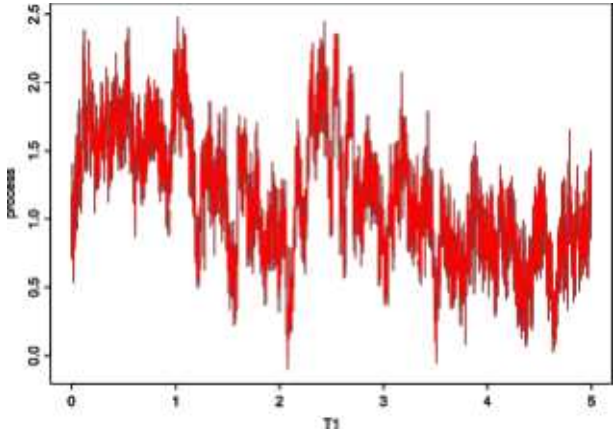}
\caption{Comparison \xch{of}{of of} $\Y$ (black) and $\tilde\Y$ (red),
$\varepsilon=0.01$. Two paths are close to each other, so that they
cannot be distinguished on the figure}\label{different Ys}
\end{figure}

\section{Equation for the square root process $Y$ for $t\ge0$}\label
{section 4}

The equation for $Y(t)$ in case of $t\in[0,\beta_0]$ has already been
obtained in Remark \ref{eq until zero}. In order to get the equation
for an arbitrary $t\in\mathbb R_+$, consider the following procedure.

Let $\mathcal I$ be the set of all non-intersecting open intervals from
representation \eqref{pos set repr}, i.e.
\begin{equation*}
\mathcal I = \bigl\{(\alpha_i, \beta_i),~i=1,\ldots, N
\bigr\}, \quad N\in\mathbb N\cup\{\infty\}.
\end{equation*}

Consider an arbitrary interval $(\alpha, \beta)\in\mathcal{I}$ and
assume that $t\in[\alpha, \beta]$. Then, from Theorem \ref{lims and eqs},
\begin{equation*}
Y(t) = \int_{\alpha}^t \frac{k}{Y(s)}ds - a\int
_{\alpha}^t Y(s) ds + \sigma \bigl(B^H(t)-B^H(
\alpha) \bigr)
\end{equation*}

Consider all such intervals $(\tilde\alpha_j, \tilde\beta_j)\in
\mathcal I$, $j =1,\ldots, M$, $M\in\mathbb N\cup\{\infty\}$, that
$\tilde\alpha_j < \alpha$.

For each $j = 1,\ldots, M$,
\begin{equation*}
0=Y(\tilde\beta_j) = \int_{\tilde\alpha_j}^{\tilde\beta_j}
\frac
{k}{Y(s)}ds - a \int_{\tilde\alpha_j}^{\tilde\beta_j} Y(s) ds +
\sigma \bigl(B^H(\tilde\beta_j) - B^H(\tilde
\alpha_j)\bigr).
\end{equation*}
Moreover,
\begin{equation*}
0=Y(\beta_0) = Y_0 + \int_{0}^{\beta_0}
\frac{k}{Y(s)}ds - a \int_{0}^{\beta_0} Y(s) ds +
\sigma B^H(\beta_0).
\end{equation*}
Therefore
\begin{equation*}
\begin{gathered} Y(t) = Y(\beta_0) + \sum
_{j=1}^M Y(\tilde\beta_j) + Y(t)
\\
= Y_0 + \Biggl(\int_{0}^{\beta_0} \biggl(
\frac{k}{Y(s)} - aY(s) \biggr)ds + \sum_{j=1}^M
\int_{\tilde\alpha_j}^{\tilde\beta_j} \biggl(\frac{k}{Y(s)} - aY(s)
\biggr)ds
\\
+\int_{\alpha}^{t} \biggl(\frac{k}{Y(s)} - aY(s)
\biggr)ds \Biggr)
\\
+\sigma \Biggl( B^H(\beta_0) + \sum
_{j=1}^M \bigl(B^H(\tilde
\beta_j) - B^H(\tilde\alpha_j) \bigr) +
\bigl(B^H(t) - B^H(\alpha) \bigr) \Biggr)
\\
= Y_0 + \int_{[0, \beta_0) \cup (\bigcup_{j=1}^M (\tilde\alpha_j,
\tilde\beta_j) )\cup[\alpha,t)} \biggl(\frac{k}{Y(s)} -
aY(s) \biggr)ds
\\
+ \sigma \Biggl(B^H(\beta_0) + \sum
_{j=1}^M \bigl(B^H(\tilde
\beta_j) - B^H(\tilde\alpha_j) \bigr) +
\bigl(B^H(t) - B^H(\alpha) \bigr) \Biggr) \\
 = Y_0 + \int_0^t
\biggl(\frac{k}{Y(s)} - aY(s) \biggr)ds
\\
+\sigma \Biggl(B^H(\beta_0) + \sum
_{j=1}^M \bigl(B^H(\tilde
\beta_j) - B^H(\tilde\alpha_j) \bigr) +
\bigl(B^H(t) - B^H(\alpha) \bigr) \Biggr). \end{gathered}
\end{equation*}

The question whether $Y$ satisfies the equation of the form \eqref{Y
until the fzh} on $\mathbb R_+$ remains open due to the obscure
structure\index{obscure structure} of its zero set. The simulation studies do not contradict
with that either.

Indeed, consider the process $Z=\{Z(t), t\ge0\}$ that satisfies the
%following SDE:
SDE
\begin{equation*}
Z(t) = Y_0 - a\int_0^t Z(s)ds +
\int_0^t \frac{k}{Y(s)}ds +\sigma
B^H(t),
\end{equation*}
where $Y$ is, as usual, the pointwise limit\index{pointwise limit} of $\Y$ and $Y_0$, $k$,
$a$, $\sigma$ are positive constants. By following explicitly the proof
of Proposition A.1 in \cite{CKM} it can be shown that
\begin{equation}
\label{expl Z model} Z(t) = e^{-a t} \Biggl(Y_0 + \int
_0^t \frac{ke^{as}}{Y(s)}ds +\sigma\int
_0^t e^{as}dB^H(s)
\Biggr).
\end{equation}

Now assume that $Y$ satisfies the equation of the form \eqref{Y until
the fzh} for all $t\ge0$. Then $Y$ admits the representation of the
form \eqref{expl Z model}, i.e.
\begin{equation}
\label{check the eq Y} Y(t) = e^{-a t} \Biggl(Y_0 + \int
_0^t \frac{ke^{as}}{Y(s)}ds +\sigma\int
_0^t e^{as}dB^H(s)
\Biggr),
\end{equation}
so the paths of $Y$ and paths obtained by direct simulation of the
right-hand side of \eqref{check the eq Y} must coincide.

In other words, if we simulate the trajectory of $Y$, apply transform
given in the RHS of \eqref{check the eq Y} to it and receive the
trajectory that significantly differs from the initial one, it would be
an evidence that $Y$ does not satisfy the equation of the form \eqref{Y
until the fzh} for all $t\ge0$.

To simulate the left-hand side of \eqref{check the eq Y}, the Euler
approximations\index{Euler approximations} of the process $\Y$ with $\varepsilon= 0.01$ are used.
They are compared to the right-hand side of \eqref{check the eq Y}
simulated explicitly using the same $\Y$. The mesh\index{mesh} of the partition is
$\Delta t = 0.0001$, $T=5$, $H=0.4$, $Y_0 = k = a = \sigma= 1$ (see
Fig. \ref{comparison fig}).

\begin{figure}[h!]
\includegraphics{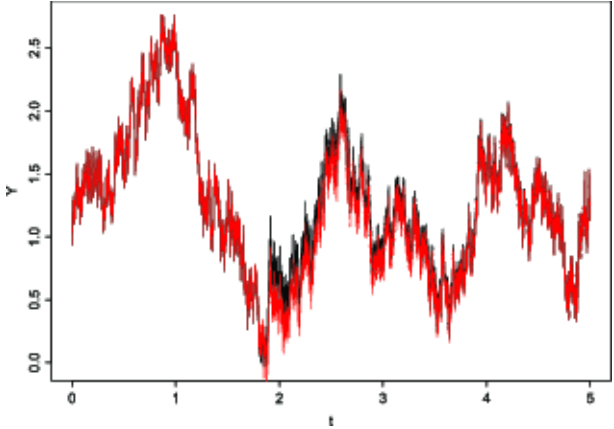}
\caption{Comparison of right-hand and left-hand sides of \eqref{check
the eq Y}}\label{comparison fig}
\end{figure}

As we see, the Euler approximation of $Y$ (of $\Y$ with small
$\varepsilon$, black) indeed almost coincides with its transformation
given by the right-hand side of \eqref{check the eq Y} (red), so no
contradiction is obtained.

\section{Fractional Cox--Ingersoll--Ross process on $\mathbb R_+$}\label
{section 5}

Consider a set of random processes $\{Y_\varepsilon, \varepsilon>0\}$\index{random processes}
which satisfy the equations of the form
\begin{equation*}
\Y(t) = Y_0 +\frac{1}{2}\int_0^t
\biggl(\fr{s}-a\Y(s) \biggr) ds +\frac
{\sigma}{2} B^H(t),
\end{equation*}
driven by the same fractional Brownian motion\index{fractional Brownian motion} with the same parameters
$k$, $a$, $\sigma>0$ and the same starting point $Y_0>0$.

Let the process $Y=\{Y(t), t \in[0,T]\}$ be such that for all $t\ge0$:
\begin{equation}
\label{square root} Y(t) = \lim_{\varepsilon\to0} \Y(t).
\end{equation}

\begin{definition}\label{newCIRdef} The fractional Cox--Ingersoll--Ross
process is the process $X=\{X(t),\allowbreak t\ge0\}$, such that
\begin{equation*}
X(t) = Y^2(t), \quad\forall t\ge0.
\end{equation*}
\end{definition}

Let us show that this definition is indeed a generalization of the
original Definition~\ref{oldCIRdef}.

First, we will require the following definition.
\begin{definition}
Let $\{U_t, t\ge0\}$, $\{V_t, t\ge0\}$ be random processes.\index{random processes}
The pathwise \emph{Stratonovich integral} $\int_0^TU_s\circ dV_s$\index{pathwise Stratonovich integral} is a
pathwise limit of the following sums
\begin{equation*}
\sum_{k=1}^{n} \frac{U_{t_{k}} +U_{t_{k-1}}}{2}
(V_{t_{k}} - V_{t_{k-1}} ),
\end{equation*}
as the mesh\index{mesh} of the partition
$0=t_0<t_1<t_2<\cdots<t_{n-1}<t_n=T$
tends to zero, in case if this limit exists.
\end{definition}

Taking into account the results of previous sections and from Theorem 1
in \cite{MYT}, for all $t\in[0,\beta_0]$:
\begin{equation*}
X(t) = X_0 + \int_0^t \bigl(k -
aX(s)\bigr)ds + \sigma\int_0^t \sqrt{X(s)}
\circ d B^H(s)~a.s.
\end{equation*}

A similar result holds for all $t\in[\alpha_i, \beta_i]$, $i\ge1$.

\begin{theorem}\label{equation for square on intervals}
Let $(\alpha_i, \beta_i)$, $i\ge1$, be an arbitrary interval from the
representation \eqref{pos set repr}\xch{. Consider}{.Consider} the fractional
Cox--Ingersoll--Ross process $X = \{X(t),~t\in[\alpha_i, \beta_i]\}$ from
Definition \ref{newCIRdef}.

Then, for $\alpha_i\le t \le\beta_i$ the process $X$ a.s. satisfies
%the following SDE:
the SDE
\begin{equation}
\begin{gathered}\label{eq: fCIR for intro} dX(t)=\bigl(k-aX(t)\bigr)dt+\sigma
\sqrt{X(t)}\circ dB^H(t), \quad X(\alpha_i)=0,
\end{gathered} %
\end{equation}
where the integral with respect to the fractional Brownian motion\index{fractional Brownian motion} is
defined as the pathwise Stratonovich integral.\index{pathwise Stratonovich integral}
\end{theorem}

\begin{proof}
We shall follow the proof of Theorem 1 from \cite{MYT}.

Let us fix an $\omega\in\varOmega$ and consider an arbitrary $t\in[\alpha
_i, \beta_i]$.

It is clear that
%
%\begin{equation}
\begin{gather}\label{eq: 1} %
X(t) = Y^2(t)\nonumber
\\
= \Biggl(\frac{1}{2}\int_{\alpha_i}^t \biggl(
\frac{k}{Y(s)}-aY(s) \biggr)ds+\frac{\sigma}{2}B^H(t) -
\frac{\sigma}{2}B^H(\alpha_i) \Biggr)^2.
\end{gather}
%\end{equation}

Consider an arbitrary partition of the interval $[\alpha_i,t]$:
\[
\alpha_i=t_0<t_1<t_2<
\cdots<t_{n-1}<t_n=t.
\]
Taking into account that $X(\alpha_i)=0$ and using \eqref{eq: 1}, we obtain
\begin{align*}
X(t) &= \sum_{j=1}^{n}
\bigl(X(t_j) - X(t_{j-1}) \bigr)
\\
&= \sum_{j=1}^n \Biggl( \Biggl[
\frac{1}{2}\int_{\alpha_i}^{t_j} \biggl(
\frac
{k}{Y(s)}-aY(s) \biggr)ds+\frac{\sigma}{2}B^H(t_j)
- \frac{\sigma
}{2}B^H(\alpha_i) \Biggr]^2
\\
&\quad- \Biggl[\frac{1}{2}\int_{\alpha_i}^{t_{j-1}}
\biggl(\frac
{k}{Y(s)}-aY(s) \biggr)ds+\frac{\sigma}{2}B^H(t_{j-1})
- \frac{\sigma
}{2}B^H(\alpha_i) \Biggr]^2
\Biggr)
\end{align*}
Factoring each summand as the difference of squares, we get:
\begin{align*}
X(t) &=\sum_{j=1}^n \Biggl[
\frac{1}{2} \Biggl(\int_{\alpha_i}^{t_{j}} \biggl(
\frac{k}{Y(s)}-aY(s) \biggr)ds+\int_{\alpha_i}^{t_{j-1}}
\biggl(\frac
{k}{Y(s)}-aY(s) \biggr)ds \Biggr)
\\
&\quad+\frac{\sigma}{2} \bigl(B_{t_j}^H +
B_{t_{j-1}}^H \bigr) - \sigma B^H(
\alpha_i) \Biggr]
\\
&\qquad\times \Biggl[\frac{1}{2}\int_{t_{j-1}}^{t_{j}}
\biggl(\frac
{k}{Y(s)}-aY(s) \biggr)ds + \frac{\sigma}{2}
\bigl(B^H(t_j) - B^H(t_{j-1})
\bigr) \Biggr].
\end{align*}
Expanding the brackets in the last expression, we obtain:
%
%\begin{equation}
\begin{gather}\label{integral sums 2} %
X(t) = \frac{1}{4}\sum
_{j=1}^n \Biggl(\int_{\alpha_i}^{t_{j}}
\biggl(\frac
{k}{Y(s)}-aY(s) \biggr)ds\nonumber
\\
+\int_{\alpha_i}^{t_{j-1}} \biggl(\frac{k}{Y(s)}-aY(s)
\biggr)ds \Biggr)\int_{t_{j-1}}^{t_{j}} \biggl(
\frac{k}{Y(s)}-aY(s) \biggr)ds\nonumber
\\
+\frac{\sigma}{4}\sum_{j=1}^n
\bigl(B_{t_j}^H + B_{t_{j-1}}^H \bigr)\int
_{t_{j-1}}^{t_{j}} \biggl(\frac{k}{Y(s)}-aY(s) \biggr)ds\nonumber
\\
-\frac{\sigma}{2} B^H(\alpha_i) \sum
_{j=1}^n \Biggl(\int_{t_{j-1}}^{t_{j}}
\biggl(\frac{k}{Y(s)}-aY(s) \biggr)ds \Biggr)\nonumber \\
 + \frac{\sigma}{4}\sum
_{j=1}^n \Biggl(\int_{\alpha_i}^{t_{j}}
\biggl(\frac
{k}{Y(s)}-aY(s) \biggr)ds+\int_{\alpha_i}^{t_{j-1}}
\biggl(\frac
{k}{Y(s)}-aY(s) \biggr)ds \Biggr)\nonumber
\\
\times \bigl(B^H(t_j) - B^H(t_{j-1})
\bigr)\nonumber
\\
+ \frac{\sigma^2}{4}\sum_{j=1}^n
\bigl(B^H({t_j}) - B^H({t_{j-1}})
\bigr) \bigl(B^H({t_j}) + B^H({t_{j-1}})
\bigr)\nonumber
\\
-\frac{\sigma^2}{2} B^H(\alpha_i) \sum
_{j=1}^n \bigl( \bigl(B^H(t_j)
- B^H(t_{j-1}) \bigr) \bigr).
\end{gather} %
%\end{equation}

Let the mesh\index{mesh} $\Delta t$ of the partition tend to zero. The first three summands
%
%\begin{equation}
\begin{gather}\label{lebesgue sums 2} %
 \frac{1}{4}\sum
_{j=1}^n \Biggl(\int_{\alpha_i}^{t_{j}}
\biggl(\frac
{k}{Y(s)}-aY(s) \biggr)ds+\int_{\alpha_i}^{t_{j-1}}
\biggl(\frac
{k}{Y(s)}-aY(s) \biggr)ds \Biggr)\nonumber
\\
\times\int_{t_{j-1}}^{t_{j}} \biggl(\frac{k}{Y(s)}-aY(s)
\biggr)ds\nonumber
\\
+\frac{\sigma}{4}\sum_{j=1}^n
\bigl(B_{t_j}^H + B_{t_{j-1}}^H \bigr)\int
_{t_{j-1}}^{t_{j}} \biggl(\frac{k}{Y(s)}-aY(s) \biggr)ds\nonumber
\\
-\frac{\sigma}{2} B^H(\alpha_i) \sum
_{j=1}^n \Biggl(\int_{t_{j-1}}^{t_{j}}
\biggl(\frac{k}{Y(s)}-aY(s) \biggr)ds \Biggr)\nonumber
\\
\rightarrow\int_{\alpha_i}^t \biggl(
\frac{k}{Y(s)}-aY(s) \biggr)\nonumber
\\
\times \Biggl(\frac{1}{2}\int_{\alpha_i}^s
\biggl(\frac{k}{Y(u)}-aY(u) \biggr)du+\frac{\sigma}{2}B^H(s)-
\frac{\sigma}{2} B^H(\alpha_i) \Biggr)ds\nonumber
\\
=\int_{\alpha_i}^t \biggl(\frac{k}{Y(s)}-aY(s)
\biggr)Y(s)ds =\int_{\alpha
_i}^t \bigl(k-aX(s) \bigr)ds,
\quad\Delta t\rightarrow0,
\end{gather}
%\end{equation}
%
and the last three summands
%
%\begin{equation}
%
\begin{gather}\label{stratonovich sums 2} %
\frac{\sigma}{4}\sum_{j=1}^n
\Biggl(\int_{\alpha_i}^{t_{j}} \biggl(\frac
{k}{Y(s)}-aY(s)
\biggr)ds+\int_{\alpha_i}^{t_{j-1}} \biggl(
\frac
{k}{Y(s)}-aY(s) \biggr)ds \Biggr)\nonumber
\\
\times \bigl(B^H(t_j) - B^H(t_{j-1})
\bigr)\nonumber
\\
+ \frac{\sigma^2}{4}\sum_{j=1}^n
\bigl(B^H({t_j}) - B^H({t_{j-1}})
\bigr) \bigl(B^H({t_j}) + B^H({t_{j-1}})
\bigr)\nonumber
\\
-\frac{\sigma^2}{2} B^H(\alpha_i) \sum
_{j=1}^n \bigl( \bigl(B^H(t_j)
- B^H(t_{j-1}) \bigr) \bigr)\nonumber \\
\rightarrow\sigma\int
_{\alpha_i}^t \Biggl(\frac{1}{2}\int
_{\alpha
_i}^s \biggl(\frac{k}{Y(u)}-aY(u) \biggr)du\nonumber
\\
+\frac{\sigma}{2}B^H(s) - \frac{\sigma}{2} B^H(
\alpha_i) \Biggr)\circ dB^H(s)\nonumber
\\
=\sigma\int_{\alpha_i}^t Y(s) \circ dB^H(s)
= \sigma\int_{\alpha_i}^t \sqrt{X(s)}\circ
dB^H(s),\quad\Delta t\rightarrow0.
\end{gather} %
%\end{equation}

Note that the left-hand side of \eqref{integral sums 2} does not depend
on the partition and the limit in \eqref{lebesgue sums 2} exists as the
pathwise Riemann integral,\index{pathwise Riemann integral} therefore the corresponding pathwise
Stratonovich integral\index{pathwise Stratonovich integral} exists and the passage to the limit in \eqref
{stratonovich sums 2} is correct.\vadjust{\eject}

Thus, the process $X$ satisfies the SDE of the form
\begin{equation}
\begin{gathered}\label{stratonovich equation 2} X(t) = \int_{\alpha_i}^t
\bigl(k-aX(s) \bigr)ds+\sigma\int_{\alpha
_i}^t
\sqrt{X(s)}\circ dB^H(s), \quad t\in[\alpha_i,
\beta_i], \end{gathered} %
\end{equation}
where $\int_{\alpha_i}^t\sqrt{X(s)}\circ dB^H(s)$ is a pathwise
Stratonovich integral.\index{pathwise Stratonovich integral}
\end{proof}

Finally, similarly to Section \ref{section 4}, the next result follows
from Theorem \ref{equation for square on intervals}.

Consider an arbitrary $(\alpha,\beta)\in\mathcal{I}$, where $\mathcal
{I}$ is a set of all open intervals from representation \eqref{pos set
repr} (see Section \ref{section 4}). Let $\beta_0$ be the first moment
of zero hitting by $Y$ and $(\tilde\alpha_j, \tilde\beta_j)\in\mathcal
{I}$, $j =1,\ldots, M$, $M\in\mathbb{N}\cup\{\infty\}$, are all such
intervals that $\tilde\alpha_j<\alpha$.

\begin{theorem}
Let $t\in[\alpha, \beta]$. Then
\begin{equation*}
\begin{gathered} X(t) = X(0) + \int_0^t
\bigl(k - a X(s)\bigr) ds
\\
+ \sigma\int_{[0, \beta_0) \cup (\bigcup_{j=1}^M (\tilde\alpha_j,
\tilde\beta_j) )\cup[\alpha,t)} \sqrt{X(s)} \circ dB^H (s),
\end{gathered} %
\end{equation*}
where
\begin{equation*}
\begin{gathered} \int_{[0, \beta_0) \cup (\bigcup_{j=1}^M (\tilde\alpha_j, \tilde
\beta_j) )\cup[\alpha,t)} \sqrt{X(s)} \circ
dB^H (s)
\\
:= \int_0^{\beta_0} \sqrt{X(s)} \circ
dB^H (s) + \sum_{j=1}^M \int
_{\tilde\alpha_j}^{\tilde\beta_j} \sqrt{X(s)} \circ dB^H (s)
\\
+ \int_{\alpha}^{t} \sqrt{X(s)} \circ dB^H
(s). \end{gathered} %
\end{equation*}
\end{theorem}

\begin{acknowledgement}[title={Acknowledgments}]
We are deeply grateful to anonymous referees whose valuable comments
helped us to improve the manuscript significantly.
\end{acknowledgement}

%\bibliography{bib/biblio}
%

\end{document}